\documentclass{amsart}
\usepackage{amssymb,amsmath,latexsym}
\usepackage{amsthm}
\usepackage{fontenc}
\usepackage{amssymb}
\numberwithin{equation}{section}

\newtheorem{theorem}{Theorem}[section]

\begin{document}
\author{Alexander E Patkowski}
\title{More on some non-gaussian wave packets}

\maketitle
\begin{abstract}We expand on a previous study, offering a generalized wave function associated with the parabolic cylinder function and a connection with a two-particle position-space wave function. We also provide an explicit formula for a wave function associated with recent work by the present author and M. Wolf. \end{abstract}

\keywords{\it Keywords: \rm Fourier Integrals; Schrodinger's equation; Free particle wave packet}

\subjclass{ \it 2010 Mathematics Subject Classification 42A38, 35Q40, 11F20.}

\section{Introduction}
In [2, 9] Fourier integrals involving the Gaussian and hyperbolic functions are discussed in detail, offering evaluations and interesting transformation properties related to Gauss sums. Recall [10] that the solution to the time-dependent Schr$\ddot{o}$dinger equation, which is $i\hslash\frac{\partial \psi(x,t)}{\partial t}=-\frac{\hslash^2}{2m}\frac{\partial^2 \psi(x,t)}{\partial x^{2}},$ $\hslash$ being Planck's constant, is known to be $\psi(x,t).$ That is, the function $\psi(x,t)$ is the Fourier transform of $\phi(z)e^{-i\hslash z^2t/2m},$ where $\phi(z)$ is a momentum amplitude. \par
In [7] we discussed the work of Glasser [4], and offered more on transformation properties as well as asymptotic expansions for special instances of $\psi(x,t).$ The paper is organized as follows. We first offer an explicit formula to a wave function (unnormalized) which had its asymptotic behavior established in [7], and accomplish this by using recent work by the present author and M. Wolf. In the following section we offer a generalization of a wave function of Glasser through using known Laplace transforms. In the last section we establish an interesting two-particle position-space wave function that may be expressed as a linear combination of the single-particle wave given in [4]. \par We shall refer to the Laplace transform of $f(z)$ throughout the paper to be, as usual,
\begin{equation}\int_{0}^{\infty}e^{-pz}f(z)dz=L_f(p),\end{equation}
and its inverse
\begin{equation}\frac{1}{2\pi i}\int_{c-i\infty}^{c+i\infty}e^{pt}L_f(p)dp=f(t), \end{equation}
and $c\in\mathbb{R}$ is chosen such that the complex integral contour is in the region where $L_f(p)$ converges.
\section{A wave function associated with generalized Glaisher-Ramanujan type integrals}
In a paper by the present author and M. Wolf [8], we gave interesting integral evaluations which generalize those found in [3]. In [7] we gave the asymptotic formula (as well as a general method) for the $b=0$ case of Theorem 2.1. Our proof will use [8, Theorem I. 1] and ideas from Ramanujan's work [9]. For the rest of this paper we shall write the primitive Dirichlet character $\chi$ to be defined as
$$\chi(n):=\begin{cases} 0,& \text {if } n=0\pmod{2},\\ 1, & \text{if } n=1\pmod{4}, \\ -1, & \text{if } n=3\pmod{4}.\end{cases}$$
All our infinite sums involving $\chi$ should be understood to be over $n\ge1$ throughout this paper.
\begin{theorem} Define $g(\alpha_2,\lambda,t):=(t\pi)^{-1/2}\int_{0}^{\infty}\sin(\alpha_1\lambda)e^{-(\alpha_1+\alpha_2)^2/(4t)}d\alpha_1.$ Put $2A(b,c)^2=\sqrt{b^4+c^2}+b^2,$ and $2B(b,c)^2=\sqrt{b^4+c^2}-b^2,$ then
\begin{equation}\int_{0}^{\infty}\cos(az)e^{-iz^2t\hslash/2m}\frac{\cosh(\frac{\pi}{2}A(b,z))\cos(\frac{\pi}{2}B(b,z))}{\cosh^2(\frac{\pi}{2} A(b,z))-\sin^{2}(\frac{\pi}{2}B(b,z))}dz\end{equation}
$$=\sum_{n}\chi(n)ng(a/\sqrt{i},(b^2+n^2)/(2\sqrt{i}),t\hslash/2m)+\sum_{n}\chi(n)ne^{-a(b^2+n^2)+i(b^2+n^2)^2t\hslash/8m}.$$
\end{theorem}
\begin{proof} First, we write down the needed integral evaluations [5]
\begin{equation}\int_{0}^{\infty}\frac{\cos(az)\cos(zy)dz}{(z^2+\beta'^2)}dz=\begin{cases} \frac{\pi}{2\beta'}e^{-a\beta'}\cosh(\beta'y),& \text {if } y<a,\\  \frac{\pi}{2\beta'}e^{-y\beta'}\cosh(\beta'a),& \text {if } a<y,\end{cases}\end{equation}
and [8, Theorem I. 1]
\begin{equation}\int_{0}^{\infty}\eta^{3}(i4x/\pi)e^{-b^2 x}\cos(cx)dx=\frac{\pi}{4}\frac{\cosh(\frac{\pi}{2}A(b,c))\cos(\frac{\pi}{2}B(b,c))}{\cosh^2(\frac{\pi}{2} A(b,c))-\sin^{2}(\frac{\pi}{2}B(b,c))},\end{equation}
where $2A(b,c)^2=\sqrt{b^4+c^2}+b^2,$ and $2B(b,c)^2=\sqrt{b^4+c^2}-b^2.$ The reader may refer to [1] for a definition of Dedekind's eta function, but for our purposes in this article we need only to note that $\eta^{3}(i4x/\pi)=\sum_{n}\chi(n)ne^{-n^2x}.$ Now Parseval's theorem for Fourier cosine transforms tells us that we may apply (2.2) and (2.3) to get that 
\begin{align}&\int_{0}^{\infty}\frac{\cos(az)}{z^2+\beta'^2}\frac{\cosh(\frac{\pi}{2}A(b,z))\cos(\frac{\pi}{2}B(b,z))}{\cosh^2(\frac{\pi}{2} A(b,z))-\sin^{2}(\frac{\pi}{2}B(b,z))}dz\\
&=\frac{4}{\pi}\sum_{n}\chi(n)n\int_{a}^{\infty}e^{-(b^2+n^2)z}\left(\frac{\pi}{2\beta}e^{-\beta'z}\cosh(a\beta')\right)dz\\
&+\frac{4}{\pi}\sum_{n}\chi(n)n\int_{0}^{a}e^{-(b^2+n^2)z}\left(\frac{\pi}{2\beta}e^{-\beta'a}\cosh(z\beta')\right)dz \\
&=\frac{4}{\pi}\sum_{n}\chi(n)n\frac{e^{-(b^2+n^2)a-\beta'a}}{\beta'+(b^2+n^2)}\left(\frac{\pi}{2\beta'}\cosh(a\beta')\right)\\
&+\frac{4}{\pi}\sum_{n}\chi(n)n\frac{e^{-a(\beta'-(b^2+n^2))}-1}{2(\beta'-(b^2+n^2))}\left(\frac{\pi}{2\beta'}e^{-a\beta'}\right)\\
&+\frac{4}{\pi}\sum_{n}\chi(n)n\frac{1-e^{-a(\beta'+(b^2+n^2))}}{2(\beta'+(b^2+n^2))}\left(\frac{\pi}{2\beta'}e^{-a\beta'}\right)\\
&=\frac{2}{\beta'}\sum_{n}\chi(n)n\left( e^{-a\beta'}\left(\frac{2(b^2+n^2)}{(2\beta')^2-(b^2+n^2)^2}\right)+e^{-a(b^2+n^2)}\left(\frac{4\beta'}{(2\beta')^2-(b^2+n^2)^2}\right)\right).\\
\end{align}
Equations (2.5)-(2.6) arise from applying the right-hand side of (2.2), while the sum over $\chi$ comes from the left side of (2.3), and this is found after applying Fourier inversion to (2.2). Equation (2.11) is found from some elementary rearrangement of (2.8)-(2.10). We may now put $\beta'^2=-ip,$ and apply the inverse Laplace transform to (2.11). To evaluate the first sum using (1.3), we need an inverse transform we were not able to locate in tables. First note the well-known evaluation [5], for $\Re(\alpha)\ge0,$
\begin{equation}\int_{0}^{\infty}t^{-1/2}e^{-\alpha^2/(4t)-pt}dt=\sqrt{\pi}p^{-1/2}e^{-\alpha p^{1/2}}.\end{equation}
Put $\alpha=\alpha_1+\alpha_2,$ then,
\begin{equation}\int_{0}^{\infty}t^{-1/2}\left(\int_{0}^{\infty}\sin(\alpha_1\lambda)e^{-(\alpha_1+\alpha_2)^2/(4t)}d\alpha_1\right)e^{-pt}dt=\lambda\sqrt{\pi}p^{-1/2}\frac{e^{-\alpha_2 p^{1/2}}}{p+\lambda^2}.\end{equation}
Hence if we set $L_g(p)$ to be the right side of (2.13), then we have our inverse $g(\alpha_2,\lambda,t):=(t\pi)^{-1/2}\int_{0}^{\infty}\sin(\alpha_1\lambda)e^{-(\alpha_1+\alpha_2)^2/(4t)}d\alpha_1.$ We now arrive at
$$\frac{1}{i}\sum_{n}\chi(n)ng(a/\sqrt{i},(b^2+n^2)/(2\sqrt{i}),t)+\frac{1}{i}\sum_{n}\chi(n)ne^{-a(b^2+n^2)+i(b^2+n^2)^2}.$$
\end{proof}

We believe our proof is interesting in that, one could make appropriate modifications to prove Ramanujan's wave function formulas [9] by using the Poisson summation formula for sine transforms.
\section{Generalized wave function and the parabolic cylinder function}
In this section we establish an interesting refinement to the non-gaussian wave function discussed in [4], which will involve the parabolic cylinder function, and therefore gives a subsequent connection with the harmonic oscillator. Recall that the parabolic cylinder function is a solution to the dimensionless form of the Schr$\ddot{o}$dinger equation over a finite range, and over the whole range is the Hermite polynomial. \par Recall that the parabolic cylinder function is given by
$$D_v(z)=2^{v/2}e^{-z^2/4}U(-\frac{v}{2},\frac{1}{2},\frac{z^2}{2}),$$ where $U(a,b,z)$ is the Tricomi confluent hypergeometric function of the first kind [1]. 

\begin{theorem} Suppose that $v$ is a complex parameter such that $v\neq 2n+1,$ $n\in\mathbb{N}.$ Then, we have that,
$$\int_{0}^{\infty} \frac{z^v}{\cosh(\beta z)}\cos(xz)e^{-i\hslash z^2t/2m}dz$$
$$=\pi2^{-v/2}\sum_{n}\chi(n)(n\pi i/2)^ve^{in^2\pi^2\hslash t/8m-nx\pi/2}$$
\begin{equation}+\pi2^{-v/2}\sum_{n}\chi(n)(i\hslash t/2m)^{-v/2-1/2}e^{i(x+in)^2m/(\hslash4 it)}D_{v}((x+in)\sqrt{m/(\hslash it)}).\end{equation}
\end{theorem}
\begin{proof}Consider the momentum space given by $\phi(z)=(z^v+(-z)^v)/\cosh(\beta z)=(e^{\pi iv}+1)z^v/\cosh(\beta z),$ where $v\neq 2n+1,$ $n\in\mathbb{N}.$ Then $\phi(z)=\phi(-z),$ and we get a Fourier cosine integral. If we put $\hslash t/2m=y$ and take the Laplace transform of the left side of (3.1) relative to $y$ we find that we have $(e^{\pi iv}+1)$ times the integral
\begin{equation}\frac{1}{2i}\int_{-\infty}^{\infty} \frac{z^v}{\cosh(\beta z) (-ip+z^2)}e^{ixz}dz.\end{equation}
If we evaluate this integral in the same manner as in [4] using the calculus of residues, we find that (3.2) is equal to 
\begin{equation} \pi\left(\frac{(-\sqrt{pi})^ve^{-(x/\sqrt{i})\sqrt{p}}}{2\sqrt{pi}\cos{(\sqrt{p/i})}}+\sum_{n}\chi(n)\frac{(n\pi i/2)^ve^{-na\pi/2}}{p-in^2\pi^2/4}\right).\end{equation}
Now from tables [5], we know that for $\Re(p)>0,$
\begin{equation}\int_{0}^{\infty}e^{-pt}t^{-v/2-1/2}e^{-\alpha/(8t)}D_{v}(\sqrt{\alpha/(2t)})dt=2^{v/2}\sqrt{\pi}p^{v/2-1/2}e^{-\sqrt{\alpha p}}.\end{equation}
So taking the inverse Laplace transform of (3.3) gives the theorem, after noting (3.4) and the inverse Laplace transform of $(p-a)^{-1}$ is $e^{at}.$
\end{proof}
Note that when $v=0$ we have Glassers' [4, eq.(9), $b=i\hslash t/2m$] after using the fact that $D_0(z)=e^{-\frac{z^2}{4}}.$ If $n\in\mathbb{N},$ then $D_n(z)=2^{-n/2}e^{-z^2/4}H_n(\frac{z}{\sqrt{2}})$ where $H_n(z)$ are the Hermite polynomials. Hence, if we select $v=2n,$ Theorem 3.1 reduces to a series involving the Hermite polynomials. \par It should be noted that another formula may be obtained using the integral formula [5], $\Re(\mu)>-1,$ $\Re(r)>0,$ $a'>0,$
$$ \int_{0}^{\infty}z^{\mu}e^{-\gamma z-rz^2}\cos(a'z)dz$$
$$=\frac{1}{2(2r)^{(\mu+1)/2}}e^{\frac{\gamma^2-a'^2}{8r}}\Gamma(\mu+1)\left(e^{-ia'\gamma/(4r)}D_{-\mu-1}((\gamma-ia')/(\sqrt{2r}))+e^{ia'\gamma/(4r)}D_{-\mu-1}((\gamma+ia')/(\sqrt{2r}))\right). $$

\section{two-particle position-space wave function}
We shall consider a two-particle positions-space wave function that has a close relationship to the single-particle wave function discussed in [4]. See [10, pg. 384] for a discussion on Multiparticle systems. We shall recall the two-particle Hamiltonian operator
$$ \hat{H}=-\frac{\hslash^2}{2m_1}\frac{\partial^2}{\partial x_1^{2}}-\frac{\hslash^2}{2m_2}\frac{\partial^2}{\partial x_2^{2}}+V(x_1,x_2).$$ (We shall only concern ourselves with potential $V(x_1,x_2)=0.$) Then we may write the Schr$\ddot{o}$dinger equation in the form $i\hslash \frac{\partial\psi}{\partial t}=\hat{H}\psi.$ We may also write the solution as 
\begin{equation}\psi(x_1,x_2,t)=\int_{-\infty}^{+\infty}\int_{-\infty}^{+\infty}\phi(z_1,z_2)e^{iz_1x_1+iz_2x_2-i\hslash z_1^2t/2m_1-i\hslash z_2^2t/2m_2}dz_1dz_2.\end{equation} The momentum-space wavefunction $\phi$ corresponding to $\psi(x_1,x_2,0)$ is given by Fourier inversion in the usual way. 
\begin{theorem} For a complex parameter $x,$ define
$$I(x):=\frac{1}{2^2}\frac{\pi\sqrt{m_1m_2}}{\hslash t}e(w_1,w_2,\hslash t/2m_1,\hslash t/2m_2)\int_{0}^{\infty}e^{iz^2(m_1+m_2)/(\hslash 2t)}\frac{\cos(xz/(t\hslash))}{\cosh(\beta'z)}dz,$$
where $$e(w_1,w_2,t_1,t_2)=e^{-(w_1^2/(4it_1)+w_2^2/(4it_2))}.$$
Then
\begin{equation} \int_{0}^{\infty}\int_{0}^{\infty}\cos(z_1w_1)\cos(z_2w_2)e^{-i\hslash z_1^2t/2m_1-i\hslash z_2^2t/2m_2}\left(\frac{\pi}{\beta'}\frac{\cosh(\frac{\pi z_1}{2\beta'})\cosh(\frac{\pi z_2}{2\beta'})}{\cosh(\frac{\pi z_1}{\beta'})+\cosh(\frac{\pi z_2}{\beta'})}\right)dz_1dz_2\end{equation} $$=I_{\beta'}(w_1m_1+w_2m_2)+I_{\beta'}(w_1m_1-w_2m_2).$$
\end{theorem}
\begin{proof} First we recall two integral formulas from tables [5]
\begin{equation} \int_{0}^{\infty} \frac{\cos(w_1z)\cos(w_1z)}{\cosh(\beta' z)}dz =\frac{\pi}{\beta'}\frac{\cosh(\frac{\pi w_1}{2\beta'})\cosh(\frac{\pi w_2}{2\beta'})}{\cosh(\frac{\pi w_1}{\beta'})+\cosh(\frac{\pi w_2}{\beta'})},\end{equation}
valid for $|\Im(w_1)|<\Re(\beta'),$ $w_2>0,$ and
\begin{equation} \int_{0}^{\infty}\cos(az)\cos(bz)e^{-\eta z^2}dz=\frac{1}{4}\sqrt{\frac{\pi}{\eta}}\left(e^{-(a-b)^2/(4\eta)}+e^{-(a+b)^2/(4\eta)}\right).
\end{equation}
Let $t_1>0,$ $t_2>0,$ then we can compute that
\begin{align}&\int_{0}^{\infty}\int_{0}^{\infty}\cos(w_1z_1)\cos(w_2z_2)e^{-it_1z_1^2-it_2z_2^2}\left(\frac{\pi}{\beta'}\frac{\cosh(\frac{\pi z_1}{2\beta'})\cosh(\frac{\pi z_2}{2\beta'})}{\cosh(\frac{\pi z_1}{\beta'})+\cosh(\frac{\pi z_2}{\beta'})}\right)dz_1dz_2\\
&=\int_{0}^{\infty}\int_{0}^{\infty}\cos(w_1z_1)\cos(w_2z_2) e^{-it_1z_1^2-it_2z_2^2}\left(\int_{0}^{\infty} \frac{\cos(z_1z)\cos(z_2z)}{\cosh(\beta' z)}dz \right)dz_1dz_2\\
&=\frac{1}{2^4}\frac{\pi }{\sqrt{t_1t_2}}\int_{0}^{\infty}\frac{(e^{-(w_1-z)^2/(4it_1)}+e^{-(w_1+z)^2/(4it_1)})(e^{-(w_2-z)^2/(4it_2)}+e^{-(w_2+z)^2/(4it_2)})}{\cosh(\beta'z)}dz\\
&=\frac{1}{2^2}\frac{\pi e(w_1,w_2,t_1,t_2)}{\sqrt{t_1t_2}}\int_{0}^{\infty}\frac{e^{-z^2(1/(4it_1)+1/(4it_2))}\cosh(w_1z/(2it_1))\cosh(w_2z/(2it_2))}{\cosh(\beta'z)}dz\\
&=\frac{1}{2^3}\frac{\pi e(w_1,w_2,t_1,t_2)}{\sqrt{t_1t_2}}\int_{0}^{\infty}\frac{e^{-z^2(1/(4it_1)+1/(4it_2))}\cosh((w_1/t_1+w_2/t_2)z/(2i))}{\cosh(\beta'z)}dz\\
&+\frac{1}{2^3}\frac{\pi e(w_1,w_2,t_1,t_2)}{\sqrt{t_1t_2}}\int_{0}^{\infty}\frac{e^{-z^2(1/(4it_1)+1/(4it_2))}\cosh((w_1/t_1-w_2/t_2)z/(2i))}{\cosh(\beta'z)}dz,
\end{align}
where $e(w_1,w_2,t_1,t_2)=e^{-(w_1^2/(4it_1)+w_2^2/(4it_2))}.$
In line (4.6) we have applied (4.3), and in line (4.7) we have applied (4.4). In line (4.9)--(4.10) we use the elementary fact that $\cosh(x)\cosh(y)=\frac{1}{2}\cosh(x+y)+\frac{1}{2}\cosh(x-y).$ Now set $t_1=\hslash t/2m_1,$ and $t_2=\hslash t/2m_2.$ Then we find that (4.9)--(4.10) is
$$=\frac{1}{2^2}\frac{\pi\sqrt{m_1m_2}}{\hslash t}e(w_1,w_2,\hslash t/2m_1,\hslash t/2m_2)\int_{0}^{\infty}\frac{e^{iz^2(m_1+m_2)/(\hslash 2t)}\cosh((w_1m_1+w_2m_2)z/(it\hslash))}{\cosh(\beta'z)}dz$$ $$+\frac{1}{2^2}\frac{\pi\sqrt{m_1m_2}}{\hslash t}e(w_1,w_2,\hslash t/2m_1,\hslash t/2m_2)\int_{0}^{\infty}e^{iz^2(m_1+m_2)/(\hslash 2t)}\frac{\cosh((w_1m_1-w_2m_2)z/(it\hslash))}{\cosh(\beta'z)}dz.$$
This is now easily seen to be equivalent to the theorem.
\end{proof}

1390 Bumps River Rd. \\*
Centerville, MA
02632 \\*
USA \\*
E-mail: alexpatk@hotmail.com

\begin{thebibliography}{9}
\bibitem{ConcreteMath} 
G. Andrews, R. Askey, and R. Roy. \emph{Special Functions,} volume 71 of Encyclopedia of Mathematics and its Applications. Cambridge University Press, New York, 1999.
\bibitem{ConcreteMath}
G. Andrews and B. C. Berndt, \emph{Ramanujan's Lost Notebook. Part IV,} Springer 2013.
\bibitem{ConcreteMath}  J.W.L. Glaisher, \emph{On the summation by definite integrals of geometric series of the second and higher order,} The Quarterly Journal of Pure and Applied Mathematics, 238--343, 1871.
\bibitem{ConcreteMath} M. L. Glasser,  \emph{Time evolution of a non-gaussian wave packet,} Physics Letters A, Volume 76A, Issues 3--4, 1980, pg. 219--220.
\bibitem{ConcreteMath}  I. S. Gradshteyn and I. M. Ryzhik. \emph{Table of Integrals, Series, and Products.} Edited by A.Jeffrey and D. Zwillinger. Academic Press, New York, 7th edition, 2007.
\bibitem{ConcreteMath} F. W. J. Olver. \emph{Asymptotics and Special Functions.} Academic Press, New York,. 1974
\bibitem{ConcreteMath} A. Patkowski, \emph{On Integrals Associated with the free particle wave packet,} Arxiv.org, June 2016.
\bibitem{ConcreteMath} A. Patkowski, M. Wolf,  \emph{Some Remarks on Glaisher-Ramanujan Type Integrals,} Computational Methods in Science and Technology v.22 (2) 103--108, (2016)
\bibitem{ConcreteMath} S. Ramanujan, \emph{Some definite integrals connected with Gauss's sums,} Messenger of Mathematics, XLIV, 1915, 75--85.
\bibitem{ConcreteMath}
Richard W. Robinett, \emph{Quantum Mechanics: Classical Results, Modern Systems, and Visualized Examples,} Second Edition, Oxford University Press, 2006.


\end{thebibliography}
\end{document}